\documentclass{amsart}
\usepackage{amsthm,amsfonts,amsmath,amssymb,latexsym,epsfig,upref,eucal,ae}
\usepackage{mathrsfs}  
\usepackage{tikz-cd}

\usepackage{url} 
\usepackage{hyperref}

\usepackage{cite}
\usepackage{scalerel,stackengine}
\usepackage{textcomp}

\usepackage{mathtools}

\usepackage{hyperref}

\usepackage{todonotes}

\theoremstyle{plain}
\newtheorem{theorem}{Theorem}[section]
\newtheorem{lemma}[theorem]{Lemma}

\newtheorem{proposition}[theorem]{Proposition}

\theoremstyle{definition}

\newtheorem{definition-theorem}[theorem]{Definition-Theorem}

\theoremstyle{remark}

\def\cF{\mathcal{F}}
\def\cE{\mathcal{E}}
\def\cO{\mathcal{O}}

\def\cG{\mathcal{G}}

\def\R{\mathbb{R}}
\def\Q{\mathbb{Q}}
\def\Z{\mathbb{Z}}
\def\N{\mathbb{N}}

\def\P{\mathbb{P}}

\def\C{\mathbb{C}}

\def\>{\rangle}

\def\<{\langle}
\def\>{\rangle}

\def\Hom{\mathrm{Hom}}
\def\Spec{\mathrm{Spec}}

\def\deg{\mathrm{deg}}

\def\Pic{\mathrm{Pic}}
\def\rank{\mathrm{rank}}
\def\dim{\mathrm{dim}}
\def\Sing{\mathrm{Sing}}
\begin{document}

\title[Stable toric sheaves of low rank]
{A note on stable toric sheaves of low rank}

\author[C. Tipler]{Carl Tipler}
\address{Univ Brest, UMR CNRS 6205, Laboratoire de Math\'ematiques de Bretagne Atlantique, France}
\email{carl.tipler@univ-brest.fr}

\date{\today}

 \begin{abstract}
  Kaneyama and Klyachko have shown that any torus equivariant vector bundle of rank $r$ over $\C\P^n$ splits if $r < n$. In particular, any such bundle is not slope stable. In contrast, we provide explicit examples of stable equivariant reflexive sheaves of rank $r$ on any polarised toric variety $(X,L)$, for $2\leq r< \dim(X)+\rank(\Pic(X))$, and show that the dimension of their singular locus is strictly bounded by $n-r$.
\end{abstract}

\maketitle

%%%%%%%%%%%%%%%%%%%%%%%%%%%%%%%%%%%%%%%%%%%%%%%%%%%%%%%%%%%%%%%%%%%%%% 
 
\section{Introduction}
\label{sec:intro}
In his study of low codimension subvarieties of $\C\P^n$, Hartshorne conjectured that for $n\geq 7$, any rank $2$ vector bundle on 
$\C\P^n$ should split (see \cite[Conjecture 6.3]{Har74}). 
While this conjecture still remains open in general, a lot of progress have been made in the equivariant context. Considering 
$\C\P^n$ as a toric variety, Kaneyama \cite{Kan} and Klyachko \cite{Kl} have shown that any torus equivariant vector bundle of rank $r<n$ over $\C\P^n$ splits as a direct sum of line bundles. More recently, Ilten and S\"uss extended this result for bundles equivariant with respect to a lower rank torus action \cite{IlSu}. As split vector bundles are not simple, they are in particular not slope stable (see Section \ref{sec:definitions} for precise definitions). 

Reflexive sheaves can be considered as midly singular versions of locally free ones \cite{Har80}. For a reflexive sheaf $\cE$ over a complex variety $X$, we denote by $\Sing(\cE)\subset X$ the singular locus of $\cE$, that is the complement in $X$ of the open set where $\cE$ is locally free. In contrast with the previously cited results, we have the following theorem.

\begin{theorem}
 \label{theo:intro}
 Let $(X,L)$ be a smooth polarised toric variety of dimension $n$ and  Picard rank $p$. Then, for any $2\leq r < n+p$, there is an equivariant stable reflexive sheaf $\cE_r$ of rank $r$ on $(X,L)$. Moreover, if $r < n$, its singular locus satisfies $\dim(\mathrm{Sing}(\cE_r)) < n - r $ and if $r\geq n$, $\cE_r$ is locally free.
\end{theorem}
Allowing for singularities provides a much greater flexibility in the constructions of (equivariant) sheaves, and the above result motivates the following question : is there a lower bound on the dimension of the singular locus of a stable equivariant reflexive sheaf of low rank on a toric variety? In Proposition \ref{prop:optimalbound}, we show that $\dim(\Sing(\cE))=n-3 $ for any rank $2$ equivariant stable sheaf $\cE$ on $\C\P^n$, $n\geq 3$. Thus in that case the bound from Theorem \ref{theo:intro} is actually 'optimal', but also the worst bound one could hope for, given that the singular locus of a reflexive sheaf on a smooth variety is always of codimension greater or equal to $3$. On the other hand, Dasgupta, Dey and Khan provided examples of rank $2$ stable equivariant bundles over specific polarised Bott towers \cite[Proposition 6.1.5]{DDK}. Thus, if such a lower bound existed, it would depend on invariants of $(X,L)$. We believe that it would be interesting to understand better the relationship between stability and singularities for (equivariant) reflexive sheaves over (toric) varieties, and this note is a first step in that direction.

\subsection*{Acknowledgments}  
The author would like to thank Achim Napame for his careful reading of the first version of the paper and his comments. The author is partially supported by the grants MARGE ANR-21-CE40-0011 and BRIDGES ANR--FAPESP ANR-21-CE40-0017.

%%%%%%%%%%%%%%%%%%%%%%%%%%%%%%%%%%%%%%%%%%%%%%%%%%%%%%%%%%%%%%%%%%%%%% 

\section{Background}
\label{sec:definitions}

Let $X$ be a smooth and complete toric variety of dimension $n$ over $\C$. We will use the standard notations from toric geometry, following \cite{CLS}. In particular, we denote by $T_N=N\otimes_\Z \C^*$ the torus of $X$, $N$ the rank $n$ lattice of its one-parameter subgroups, $M=\Hom_\Z(N,\Z)$ its character lattice and $\Sigma$ its fan of strongly convex rational polyhedral cones in $N_\R=N\otimes_\Z \R$ (see \cite[Chapter 3]{CLS}). 
The variety $X$ is then covered by the $T_N$-invariant affine varieties $U_\sigma=\Spec(\C[M\cap\sigma^\vee])$, for $\sigma\in \Sigma$.
\subsection{Equivariant sheaves}
Let $\alpha : T_N \times X \to X$,  $\pi_1 : T_N\times X \to T_N$ and $\pi_2 : T_N \times X \to X$ be the $T_N$-action, the projection on $T_N$ and the projection on $X$ respectively. 
A coherent sheaf $\cF$ on $X$ is {\it $T_N$-equivariant} (or {\it equivariant} for short) if there is an isomorphism
$\varphi : \alpha^*\cF \to \pi_2^*\cF$
satisfying some cocycle condition (see for example \cite[Section 5]{perl}). Klyachko provided a simple description of equivariant reflexive sheaves on toric varieties \cite{Kl} (see also \cite{perl}). Any such sheaf is uniquely described by a {\it family of filtrations} that we denote $(E,E^\rho(i))_{\rho\in\Sigma(1),i\in\Z}$. Here, $E$ is a finite dimensional complex vector space of dimension $\rank(\cE)$, and for each ray $\rho\in\Sigma(1)$, $(E^\rho(i))_{i\in\Z}$ is a bounded {\it increasing} filtration of $E$ (note here that we will use increasing filtrations as in \cite{perl}, rather than decreasing ones as in \cite{Kl}). Then, one recovers an equivariant reflexive sheaf $\cE$ by setting for each $\sigma\in\Sigma$ :
\begin{equation*}
  \label{eq:sheaf from family of filtrations}
  \Gamma(U_{\sigma}, \cE):=\bigoplus_{m\in M} \bigcap_{\rho\in\sigma(1)} E^\rho(\langle m,u_\rho\rangle)\otimes \chi^m
 \end{equation*}
 where $u_\rho\in N$ is the primitive generator of $\rho$ and $\langle \cdot,\cdot\rangle$ the duality pairing. Finally, from \cite[Theorem 2.2.1]{Kl} (or \cite[Section 5]{perl}),  $\cE$ will be locally free  if and only if the family of filtrations $(E^\rho(\bullet))_{\rho\in\Sigma(1)}$ satisfies Klyachko's compatibility criterion, namely that for each $\sigma\in\Sigma$, there exists a decomposition 
$$
E= \bigoplus_{[m]\in M/(M\cap\sigma^\perp)} E^\sigma_{[m]}
$$
such that for each ray  
$\rho\in\sigma(1)$ in $\sigma$ :
$$
E^\rho(i)=\bigoplus_{\langle m,u_\rho\rangle\leq i}  E^\sigma_{[m]}.
$$
\subsection{Slope stability}
Assume now that $L\to X$ is an ample line bundle on $X$. Recall that a reflexive sheaf $\cE$ on $X$ is said to be {\it slope stable} if for any coherent and saturated subsheaf $\cF\subset \cE$ with $\rank(\cF)<\rank(\cE)$, one has
$$
\mu_L(\cF)<\mu_L(\cE),
$$
where for any coherent torsion-free sheaf $\cG$, the {\it slope} $\mu_L(\cG)$ is the intersection number
$$
\mu_L(\cG)=\frac{c_1(\cG)\cdot L^{n-1}}{\rank(\cG)}\in\Q.
$$
If $\cE$ is equivariant with associated family of filtrations $(E,E^\rho(\bullet))_{\rho\in\Sigma(1)}$, from Klyachko's formula for the first Chern class (see e.g. \cite[Corollary 2.18]{ClaTip}) we obtain  
\begin{equation}
 \label{eq:slope}
 \mu_L(\cE)=-\frac{1}{\rank(\cE)}\sum_{\rho\in\Sigma(1)} \iota_\rho(\cE)\, \deg_L(D_\rho),
\end{equation}
 where $\deg_L(D_\rho)$ is the degree with respect to $L$ of the divisor $D_\rho$ associated to the ray $\rho\in\Sigma$, and where
 $$
 \iota_\rho(\cE):=\sum_{i\in\Z} i \left(\dim(E^\rho(i))-\dim (E^\rho(i-1))\right).
 $$
Moreover, in that equivariant case, from Kool's work \cite[Proposition 4.13]{Koo} (see also \cite[Proposition 2.3]{HNS}), to check stability for $\cE$, it is enough to compare slopes with equivariant and saturated reflexive subsheaves. By \cite[Lemma 2.14]{NaTip}, any such subsheaf is associated to a family of filtrations of the form $(F, F\cap E^\rho(i))_{\rho\in\Sigma(1),i\in\Z}$ for some vector subspace $F\subsetneq E$. To summarize, we have
\begin{proposition}
\label{prop:stability equiv subsheaves}
The equivariant reflexive sheaf associated to the family of filtrations $(E,E^\rho(i))_{\rho\in\Sigma(1), i\in\Z}$ is slope stable if and only if for any vector subspace $F\subsetneq E$, we have 
$$
\frac{1}{\dim(F)}\sum_{\rho\in\Sigma(1)} \iota_\rho(F)\, \deg_L(D_\rho) > \frac{1}{\dim(E)}\sum_{\rho\in\Sigma(1)} \iota_\rho(\cE)\, \deg_L(D_\rho),
$$
where 
$$
 \iota_\rho(F):=\displaystyle\sum_{i\in\Z} i \left(\dim(F\cap E^\rho(i))-\dim (F\cap E^\rho(i-1))\right).
$$
 \end{proposition}
\section{The examples}
\label{sec:examples}
Once the above settled, the proof of Theorem \ref{theo:intro} is fairly simple and relies on elementary observations. To produce the examples, we will need the following lemma.
\begin{lemma}
 \label{lem:independentvectors}
 Let $(r,m)\in \N^2$ with $r\geq 2$ and $m\geq 1$. There exists $(v_i)_{1\leq i\leq m} \in(\C^r)^m$ such that for any $d\leq \min\lbrace r, m \rbrace$ and any $\lbrace i_1,\ldots, i_d\rbrace \subset \lbrace 1,2,\ldots,m\rbrace$, the vectors $(v_{i_1}, \ldots , v_{i_d})$ are linearly independent.
\end{lemma}
\begin{proof}
 If  $m\leq r$ the statement is obvious. For $m\geq r$, we use induction on $m$. Assuming that we have $(v_i)_{1\leq i\leq m} \in(\C^r)^m$ satisfying the conclusion of the lemma, we can pick $v_{m+1}\in\C^r$ in the complementary of the finite union of hyperplanes defined by the equations 
 $$
 \det(v_{i_1},\ldots,v_{i_{r-1}}, x)=0
 $$
 where the set of indices $\lbrace i_1,\ldots , i_{r-1}\rbrace$ runs through all subsets of $  \lbrace 1, \ldots, m \rbrace$ with $r-1$ elements.
\end{proof}
\begin{proof}[Proof of Theorem \ref{theo:intro}]
 Let $r\in\N$ with $1 < r < n+p$, where we recall that $p=\rank(\Pic(X))$. From \cite[Theorem 4.2.1]{CLS}, $n+p=\vert \Sigma(1)\vert$ is the number of rays in $\Sigma$, so we have 
 $$
 2\leq r \leq \vert \Sigma(1)\vert-1.
 $$
 By Lemma \ref{lem:independentvectors}, we can fix $m=\vert \Sigma(1)\vert $ vectors $(v_\rho)_{\rho\in\Sigma(1)}$ in $\C^r$ such that any $d$-dimensional subspace $F\subset \C^r$ contains at most $d$ elements from $\lbrace v_\rho,\: \rho\in\Sigma(1) \rbrace$. We also set for $\rho\in\Sigma(1)$ :
 $$
 m_\rho = \displaystyle \underset{\rho'\neq \rho}{\Pi} \deg_L(D_{\rho'})\in\N^*,
 $$
so that there is a positive constant $c\in\N^*$ such that for any $\rho\in\Sigma(1)$,
$$
m_\rho\, \deg_L(D_\rho) =c.
$$
 Now, define $\cE_r$ to be the equivariant reflexive sheaf associated to the family of filtrations $(\C^r,E^\rho(i))_{\rho\in\Sigma(1),i\in\Z}$ with 
 $$
 E^\rho(i)=\left\{ 
 \begin{array}{ccc}
  \lbrace 0 \rbrace & \mathrm{ if } & i < 0 \\
  \C\cdot v_\rho & \mathrm{ if } & 0\leq  i < m_\rho \\
  \C^r & \mathrm{ if } & m_\rho \leq i . 
 \end{array}
\right. 
 $$
 By construction, using formula (\ref{eq:slope}), we have 
 $$
 \begin{array}{ccc}
- \mu_L(\cE_r) & = & \displaystyle\frac{1}{r}\sum_{\rho\in\Sigma(1)} (r-1)\,m_\rho\, \deg_L(D_\rho)\\
 & = & \displaystyle c\, \frac{r-1}{r}\vert \Sigma(1)\vert.
 \end{array}
 $$
 On the other hand, for a $d$-dimensional subspace $F\subsetneq \C^r$, we compute 
 $$
 \iota_\rho(F)\, \deg_L(D_\rho)=\left\{ \begin{array}{ccc}
                  c\,(d-1)& \mathrm{ if } & v_\rho \in F \\ 
                   c\,d & \mathrm{ if } & v_\rho \notin F
                   \end{array}
\right. 
$$
and then 
$$
 \begin{array}{ccc}
\displaystyle\frac{1}{\dim(F)}\sum_{\rho\in\Sigma(1)} \iota_\rho(F)\, \deg_L(D_\rho) & = & \displaystyle c\,\vert \Sigma(1) \vert - \sum_{v_\rho\in F} \frac{c}{d} \\
\displaystyle& \geq &  c\,\vert\Sigma(1)\vert - c
 \end{array}
$$
where the last inequality comes from the fact that $F$ contains at most $d$ elements amongst $(v_\rho)_{\rho\in\Sigma(1)}$. As $r < \vert \Sigma(1) \vert=n+p$, we then conclude that 
$$
- \mu_L(\cE_r) < \displaystyle\frac{1}{\dim(F)}\sum_{\rho\in\Sigma(1)} \iota_\rho(F)\, \deg_L(D_\rho)
$$
and by Proposition \ref{prop:stability equiv subsheaves}, $\cE_r$ is slope stable. We now turn to the singular locus $\Sing(\cE_r)\subset X$. Assume first that $r\leq n$. Let $\sigma\in \Sigma(r)$ a $r$-dimensional cone in $\Sigma$. As $X$ is smooth, we can find an isomorphism $N\simeq \Z^n$ such that the elements $(e_1,\ldots, e_r)$ from the canonical basis of $\Z^n$ span the rays $(\rho_1,\ldots, \rho_r)$ of $\sigma(1)$. We can then identify 
$$
M/(M\cap\sigma^\perp)\simeq \Z\cdot e_1^*\oplus\ldots \oplus \Z\cdot e_r^*
$$
for $(e_i^*)_{1\leq i\leq n}$ the dual canonical basis. For $j\in [\![ 1, r ]\!]$, define $E^\sigma_j$ to be the vector space $\C\cdot v_{\rho_j}\subset \C^r$ together with the weight $\mu^j$-action of $T_N$, where
$$
\mu^j:=\sum_{1\leq i\leq r, \, i\neq j} m_{\rho_i}\,e_i^*\in \Z\cdot e_1^*\oplus\ldots \oplus \Z\cdot e_r^*.
$$
Then, by choice of the $(v_\rho)_{\rho\in\Sigma(1)}$, we have 
$$
\C^r=\bigoplus_{j=1}^r E^\sigma_j,
$$
and by choice of the weights $(\mu^j)_{1\leq j\leq r}$, for any $\rho_k\in \sigma(1)$, we infer that 
$$
E^{\rho_k}(i)= \bigoplus_{\mu^j_k\leq i}  E^\sigma_j
$$
where $\mu^j_k$ stands for the $k$-th coordinate of $\mu^j$ in $\bigoplus_{i=1}^n \Z\cdot e_i^*$. Thus, Klyachko's criterion for locally freeness is satisfied by $(\cE_r)_{\vert U_\sigma}$, the restriction of $\cE_r$  to $U_\sigma$. Hence, we have 
$$
\Sing(\cE_r)\subset  X \setminus \bigcup_{\sigma\in\Sigma(r)} U_\sigma.
$$
By the orbit-cone correspondence (see \cite[Theorem 3.2.6]{CLS}), we deduce that 
$$
\Sing(\cE_r)\subset \bigcup_{\dim(\tau)> r} \cO(\tau)
$$
where $\cO(\tau)$ is the $(n-\dim(\tau))$-dimensional orbit associated to $\tau\in\Sigma$. Hence, 
$$
\dim(\Sing(\cE_r)) < n-r.
$$
The case for $r > n$ can be dealt with a similar argument, and this concludes the proof.
\end{proof}
In view of Hartshorne's conjecture, it is natural to try to build rank $2$ stable reflexive sheaves  with a singular locus of the smallest possible dimension. Unfortunately, on $\C\P^n$, the bound $\dim(\Sing(\cE_2))\leq n-3$ is actually optimal in the torus equivariant case.
\begin{proposition}
 \label{prop:optimalbound}
 Let $\cE$ be a rank $2$ stable equivariant reflexive sheaf on $\C\P^n$, $n\geq 3$. Then $\dim(\Sing(\cE))=n-3$.
\end{proposition}

\begin{proof}
In the proof, we specify to $X=\C\P^n$, but keep all previous notations (e.g. $\Sigma$ will denote the fan of $\C\P^n$).  Consider $(E,E^\rho(\bullet))_{\rho\in\Sigma(1)}$ the family of filtrations associated to $\cE$. Up to an isomorphism, we can assume $E=\C^2$. For each $\rho\in\Sigma(1)$, there exist integers $n_\rho \leq m_\rho$ and a vector $v_\rho\in\C^2$ such that 
 $$
 E^\rho(i)=\left\{
 \begin{array}{ccc}
                  \lbrace 0 \rbrace & \mathrm{ if } & i< n_\rho \\
                  \C\cdot v_\rho & \mathrm{ if } & n_\rho\leq i< m_\rho \\
\C^2 & \mathrm{ if } & m_\rho\leq i .
                  \end{array}
                  \right.
 $$
Up to tensoring $\cE$ by $\cO(\sum_{\rho}n_\rho\, D_\rho)$, we can further assume that $n_\rho=0$ for all $\rho\in\Sigma(1)$, see \cite[Remark 2.2.15]{DDK}. We then have 
$$
 E^\rho(i)=\left\{
 \begin{array}{ccc}
                  \lbrace 0 \rbrace & \mathrm{ if } & i< 0 \\
                  \C\cdot v_\rho & \mathrm{ if } & 0\leq i< m_\rho \\
\C^2 & \mathrm{ if } & m_\rho\leq i .
                  \end{array}
                  \right.
 $$
 If for all $\rho\in\Sigma(1)$, $m_\rho=0$, then $\cE\simeq \cO_{\P^n}\oplus\cO_{\P^n}$, which contradicts stability. Hence, there is at least one line $\C\cdot v_\rho$ appearing in the family of filtrations of $\cE$. We claim now that there must be at least three different such lines. Indeed, if for all $\rho\in\Sigma(1)$ with $m_\rho >0$ we have $\C\cdot v_\rho = F_1$ for a given line $F_1\subset \C^2$, then we have on one hand (we assume $\deg_L(D_\rho)=1$ for all $\rho\in\Sigma(1)$) : 
 $$
 -\mu_L(\cE)=\frac{1}{2}\sum_{m_\rho\neq 0} m_\rho >0 ,
 $$
 while 
 $$
 \sum_{\rho\in\Sigma(1)} \iota_\rho(F_1) = 0,
 $$
 which contradicts stability by Proposition \ref{prop:stability equiv subsheaves}. If there are two different lines $F_1, F_2\subset \C^2$ such that for any $\rho\in\Sigma(1)$ with $m_\rho >0$, $\C\cdot v_\rho = F_1$ or $\C\cdot v_\rho = F_2$, then 
 $$
 -\mu_L(\cE)=\frac{1}{2}\left(\sum_{\C\cdot v_\rho= F_1} m_\rho+\sum_{\C\cdot v_\rho= F_2} m_\rho\right),
 $$
 while
  $$
 \sum_{\rho\in\Sigma(1)} \iota_\rho(F_1) = \sum_{\C\cdot v_\rho= F_2} m_\rho
 $$
 and 
 $$
 \sum_{\rho\in\Sigma(1)} \iota_\rho(F_2) = \sum_{\C\cdot v_\rho= F_1} m_\rho.
 $$
 Again, this contradicts stability. Hence, we can find at least three different lines $F_1, F_2, F_3 \subset \C^2$ together with $(\rho_1, \rho_2, \rho_3)\in(\Sigma(1))^3$ such that for $1\leq i\leq 3$, $\C\cdot v_{\rho_i} = F_i$. As $n\geq 3$, the cone $\sigma:=\sum_{i=1}^3 \rho_i$ belongs to $\Sigma$. Klyachko's criterion for locally freeness cannot be satisfied on $U_\sigma$, as it would imply that $F_1\oplus F_2\oplus F_3 \subset \C^2$, which is absurd. It follows that $\Sing(\cE)\cap U_\sigma \neq \varnothing$. On the other hand, arguing as in the proof of Theorem \ref{theo:intro}, we have that for any face $\tau \subsetneq \sigma$, $\cE$ is locally free on $U_\tau \subset U_\sigma$. We conclude by invariance of the singular locus and the orbit-cone correspondence that $\cO(\sigma)\subset \Sing(\cE)$, and as $\dim(\cO(\sigma))=n-3$, $\dim(\Sing(\cE))\geq n-3$. The result then follows from general theory, the singular locus of a reflexive sheaf on a smooth complex manifold being always of codimension at least $3$.
\end{proof}

\bibliography{lowrank}
\bibliographystyle{amsplain}

\end{document}